\documentclass[letterpaper,12pt]{article}
\usepackage[utf8x]{inputenc}
\usepackage[title]{appendix}
\usepackage[noTeX]{mmap}
\usepackage[left=1.3in, right=1.3in, bottom = 1.4in, top = 1.4in]{geometry}
\usepackage{hyperref} 
\hypersetup{colorlinks}
\usepackage{amsmath,amsfonts, amsthm, amssymb}
\usepackage{color}
\usepackage{enumerate}
\usepackage{hypbmsec}
\usepackage{bbm}
\usepackage{dsfont}
\usepackage{mathtools}
\usepackage{graphicx}
\usepackage{caption}
\usepackage{subcaption}
\usepackage[section]{placeins}
\usepackage{tikz}
\usetikzlibrary{calc}
\usetikzlibrary{decorations.pathreplacing}
\usetikzlibrary{arrows}
\usepackage{pgffor}
\usepackage{longtable}
\usepackage{comment}
\usepackage{multirow}
\usepackage{rotating}
\usepackage{bm}
\usepackage{fancyhdr}
\usepackage{tocloft}

\theoremstyle{plain}
\newtheorem{theorem}{Theorem}[section]

\newtheorem{corollary}[theorem]{Corollary}
\newtheorem{lemma}[theorem]{Lemma}
\newtheorem{definition}[theorem]{Definition}
\newtheorem{example}[theorem]{Example}
\newtheorem{conjecture}[theorem]{Conjecture}

\theoremstyle{definition}
\newtheorem{remark}[theorem]{Remark}

\renewcommand{\P}{\mathbb{P}}
\newcommand{\Z}{\mathbb{Z}}
\newcommand{\N}{\mathbb{N}}
\newcommand{\T}{\mathcal{T}}
\renewcommand{\r}{\mathsf{r}}
\renewcommand{\d}{\mathrm{d}}

\newcommand{\E}{\mathbb{E}}

\newcommand{\Pb}{\mathbb{P}}

\newcommand{\D}{\mathcal{D}}

\newcommand{\indicator}{\mathds{1}}

\newcommand{\G}{\mathsf{G}}

\newcommand{\Tgood}{\mathsf{T}_i^{\mathsf{good}}}

\newcommand{\pref}[1]{\hyperref[#1]{[p. \pageref{#1}]}}

\newcommand{\old}[1]{}

\DeclareMathOperator{\Var}{Var}

\title{A law of large numbers for the range of rotor walks on periodic trees}

\author{Wilfried Huss and Ecaterina Sava-Huss}

\date{\today}

\begin{document}
\maketitle

\begin{abstract}
The aim of the current work is to prove a law of large numbers for the range size of recurrent rotor walks with random initial configuration on a general class of trees, called \emph{periodic trees} or \emph{directed covers of graphs}. This generalizes \cite[Theorem 1.1]{huss_sava_range_speed_trees},
but the proofs are of different nature and rely on generating functions.
\end{abstract}



\section{Introduction}

In \cite{huss_sava_range_speed_trees}, we have considered rotor walks $(X_n)_{n\geq 0}$ with random initial configuration of rotors on regular trees and on Galton-Watson trees, and we have proven a law of large numbers for the size of the range and for the rate of escape. The proofs used in \cite{huss_sava_range_speed_trees}
are mostly based on analysing the behavior and the growth of certain branching processes that appear when looking at rotor walks on trees. We continue the investigation of the range of rotor walk in this work, and we prove similar results for positive recurrent rotor walks on periodic trees, but with completely different methods, that are based on generating functions and on linear algebra.

 Periodic trees are a straightforward generalization of regular trees.  With the methods we introduce here, we could also recover parts of the results in \cite{huss_sava_range_speed_trees}. Nevertheless, the methods for periodic trees are technically more involved, and give the results in terms of spectral radii of adjacency matrices related to the initial periodic tree and initial random rotor configuration. Transience and recurrence of rotor walks on such trees was investigated in \cite{huss_sava_trans_directed_covers}. 

Before stating the main result, we introduce shortly the setting we are working on. 
Let $\mathsf{G}$ be a finite graph on $\mathsf{N}$ vertices, $D$ its adjacency matrix, and $\mathsf{T}_i$ be a periodic tree (directed cover of $\mathsf{G}$), with root of type $i\in\{1,\ldots,\mathsf{N}\}$. A rotor configuration on $\mathsf{T}_i$ is a function which assigns to each vertex a rotor pointing to one of the neighbors, and the neighbors of each vertex are ordered counterclockwise. A rotor walk $(X_n^i)_{n\geq 0}$ on $\mathsf{T}_i$ is a process which starts at the root vertex, and at each time step it does the following: it first rotates the rotor to point to the next neighbor in the counterclockwise order, and then it moves to the neighbor the rotor points at. A child of a vertex is called good if the rotor walk visits this child before returning to the parent vertex. The tree of good children of $(X_n^i)$, which we denote by  $\Tgood$, is a subtree of $\mathsf{T}_i$, consisting of only good children. Suppose that $(X_n^i)$ is a rotor walk on $\mathsf{T}_i$, with $\mathcal{D}$-distributed random initial configuration of rotors; see the paragraph above Definition \ref{defn:good_children} for the precise definition of the distribution $\mathcal{D}$. Then $\Tgood$  is a \emph{multitype branching process} (MBP), whose first moment matrix will be denoted by $M$. It has been shown in \cite{huss_sava_trans_directed_covers} that $(X_n^i)$, for every $i\in\{1,\ldots,\mathsf{N}\}$ is recurrent if and only if the spectral radius $\rho(M)\leq 1$. The range $R_n^i=\{X_1^i,\ldots,X_n^i\}$ represents the number of distinct visited points by the walk $(X_n^i)$ up to time $n$ and $|R_n^i|$ its size. The walk $(X_n^i)$ starts at the root of $\mathsf{T}_i$, and the super index gives the dependence on the type $i$ of root vertex.
For the size $|R_n^i|$ of the rotor walk up to time $n$ we prove the following law of large numbers.
\begin{theorem}\label{thm:main-thm}
Let $(X_n^i)_{n\geq 0}$ be a  rotor walk with a $\mathcal{D}$-distributed random initial configuration of rotors on a periodic tree $\mathsf{T}_i$, with root of type $i\in\{1,\ldots,\mathsf{N}\}$.
For all $i\in\{1,\ldots,\mathsf{N}\}$, if $\rho(M)<1$, then 
\begin{equation*}
\frac{|R_n^i|}{n}\to \frac{1}{2}\left(1-\frac{1}{\gamma}\right),\quad \text{almost surely, as  } n\to \infty,
\end{equation*}
where $\gamma$ is the spectral radius of the matrix $I+(D-I)(I-M)^{-1}$.
\end{theorem}
In the null recurrent case, when $\rho(M)=1$, we conjecture the following.
\begin{conjecture}
For all $i\in\{1,\ldots,\mathsf{N}\}$, if $\rho(M)=1$, then 
\begin{equation*}
\frac{|R_n^i|}{n}\to \frac{1}{2},\quad\text{ almost surely, as } n\to\infty.
\end{equation*}
\end{conjecture}
Let us briefly comment on the proof of Theorem \ref{thm:main-thm}.  For $\rho(M)<1$, the methods involved use generating functions equalities, multitype Galton-Watson processes, and are quite technical in comparison to the ones used on regular trees in  \cite{huss_sava_range_speed_trees}. For the other two cases (when the rotor walk is null recurrent ($\rho(M)=1$) and transient ($\rho(M)>1$)) one can
also prove a law of large numbers for the size of the range, but the amount of technical calculations will be too high for the expected result. Range of rotor walks and its shape was considered also in \cite{florescu_levine_peres} on comb lattices and on Eulerian graphs. It is conjectured in \cite{KD}, that on $\Z^2$, the range of uniform rotor walks is asymptotically a disk, and its size if of order $n^{2/3}$. This is probably hard to achieve, since on $\Z^2$ is not even known whether the uniform rotor walk is recurrent or transient.
For recent results  on rotor walks on transient graphs with initial rotor configuration sampled from the wired uniform spanning forest oriented toward infinity measure see \cite{swee-hong-wsf,swee-hong-wsf1}. See also \cite{swee-escape-rate} where for any given graph, a rotor configuration with maximum escape rate is constructed.

\section{Preliminaries}

Since we use the results from \cite{huss_sava_trans_directed_covers} on recurrence/transience of rotor walks on periodic trees, we keep the same notation and definitions as used there.

\subsection{Periodic trees}

\emph{Periodic trees} are also known in the literature as \emph{directed covers of graphs}
or as \emph{trees with finitely many cone types}. Such trees have interesting properties in what concerns both the behavior of random walks and of rotor walks on them. We add in the appendix of the paper an example of a rotor-recurrent periodic tree that contains rotor-transient subtrees.

\paragraph{Graphs and Trees.}
Let $\G=(V,E)$ be a locally finite and connected directed multigraph, with vertex set $V$
and edge set $E$. For ease of presentation, we shall identify the graph $\G$ with its vertex 
set $V$, i.e., $i\in \G$ means $i\in V$. If $(i,j)$ is an edge of $\G$, we write $i\sim_{\G}j$,
and write $\d(i,j)$ for the {\em graph distance}.
Let $D=(d_{ij})_{i,j\in\G}$ be the \emph{adjacency matrix} of $\G$, where $d_{ij}$ is
the number of directed edges connecting $i$ to $j$.
We write $d_i$ for the sum of the entries in the $i$-th row of $D$, that is
$d_{i}=\sum_{j\in\G} d_{ij}$ is the {\em degree} of the vertex $i$. 

A {\em tree} $\T$ is a connected, cycle-free graph. A {\em rooted tree} is a tree with a distinguished vertex
$r$, called {\em the root}. For a vertex $x\in\T$, denote by $|x|$ the {\em height} of $x$, that is the graph distance from the root to $x$.
For  $x\in \T\setminus\{r\}$, denote by $x^{(0)}$ its {\em ancestor}, that is
the unique neighbor of $x$ closer to the root $r$.
We attach to $\T$ an additional vertex $o$ to the root $r$, which will
be considered in the following as a sink vertex.
Additionally we fix a planar embedding of $\T$ and enumerate
the neighbors of a vertex $x\in \T$ in counterclockwise order $\big(x^{(0)}, x^{(1)}, \ldots, x^{(d_x-1)}\big)$
beginning with the ancestor.
We will call a vertex $y$ a \emph{descendant} of $x$, if $x$ lies on the unique shortest path from $y$
to the root $r$. A descendant of $x$, which is also a neighbor of $x$, will be called a \emph{child}.
The \emph{principal branches} of $\T$ are the subtrees rooted at the children of the root $r$.

\paragraph{Directed Covers of Graphs.}
Suppose now that $\G$ is a finite, directed and strongly connected multigraph with  adjacency matrix $D=(d_{ij})$. 
Let $\mathsf{N}$ be the cardinality of the vertices of $\G$, and label the vertices of $\G$ by $\{1,2,\ldots, \mathsf{N} \}$.
The {\em directed cover} $\mathsf{T}$ of $\G$ is defined recursively as a rooted tree
 whose vertices are labeled by the vertex set $\{1,2,\ldots, \mathsf{N} \}$ of $\G$.
The root $r$ of $\mathsf{T}$ is labeled with some $i\in\{1,2,\ldots, \mathsf{N} \}$. Recursively, if $x$ is a vertex in $\mathsf{T}$ with
label $i\in \G$, then $x$ has $d_{ij}$ descendants with label $j$. We define
the {\em label function} $\tau:\mathsf{T}\to\G$ as the map that associates to each vertex in $\mathsf{T}$ its label 
in $\G$. The label $\tau(x)$ of a vertex $x$ will be also called the {\em type} of $x$. 
For a vertex $x\in\mathsf{T}$, we will not only need its type, but also the types of its children.
In order to keep track of the type of a vertex and the types of its children we
introduce the {\em generation function} $\chi = \left(\chi_i\right)_{i\in\G}$ with 
$\chi_i:\{1,\ldots,d_i\}\to\G$. For a vertex $x$ of type $i$,
$\chi_i(k)$ represents the type of the $k$-th child $x^{(k)}$ of $x$, i.e.,
\begin{equation*}
\text{if } \tau(x)=i \text{ then } \chi_i(k)=\tau(x^{(k)}), \text{ for } k=1,\ldots,d_i.
\end{equation*}
As the neighbors $\left(x^{(0)},\ldots, x^{(d_{\tau(x)})}\right)$ of any vertex $x$ are drawn in counterclockwise order,
the generation function $\chi$ also fixes the planar embedding of the tree and thus defines $\mathsf{T}$ uniquely
as a planted plane tree. The tree $\mathsf{T}$ constructed
in this way is called the \emph{directed cover} of $\G$. Such trees are also known as
\emph{periodic trees}, see \cite{lyons-peres-book}.
The graph $\G$ is called the \emph{base graph} or the \emph{generating graph} for the tree $\mathsf{T}$.
We write $\mathsf{T}_i$ for a tree with root $r$ of type $i$, that is $\tau(r)=i$, and we say that 
$\mathsf{T}_i$ is a periodic tree with $\mathsf{N}$ types of vertices and root of type $i\in\{1,2,\ldots,\mathsf{N}\}$.

\subsection{Multitype branching processes.}

A multitype branching process (MBP) is a generalization of a Galton-Watson process,
where  one allows a finite number of distinguishable types of particles with different
probabilistic behavior. The particle types will coincide with the different types of vertices
in the periodic trees under consideration, and will be denoted by $\{1,\ldots,\mathsf{N}\}$.

A \emph{multitype branching process} is a Markov process $(\mathsf{Z}_n)_{n\in\N_0}$
such that the states $\mathsf{Z}_n = (\mathsf{Z}_{n,1},\ldots,\mathsf{Z}_{n,\mathsf{N}})$ are $\mathsf{N}$-dimensional vectors with non-negative components. The initial state
$\mathsf{Z}_0$ is nonrandom. The $i$-th entry $\mathsf{Z_{n,i}}$ of $\mathsf{Z}_n$ represents the number of particles of type $i$ in the 
$n$-th generation. The transition law of the process is as follows. If $\mathsf{Z}_0=\mathsf{e}_i$,
where $\mathsf{e}_i$ is the $\mathsf{N}$-dimensional vector whose $i-th$ component is $1$ and all the
others are $0$, then $\mathsf{Z}_n$ has the generating function
$\mathsf{f}(\mathsf{z}) = \big(f^1(\mathsf{z}),\ldots,f^{\mathsf{N}}(\mathsf{z})\big)$ with
\begin{equation}
\label{eq:MBP_generating_function}
 f^i(\mathsf{z}) = f^i(z_1,\ldots,z_{\mathsf{N}})
 = \smashoperator{\sum_{s_1,\ldots,s_{\mathsf{N}} \geq 0}} p^i(s_1,\ldots,s_{\mathsf{N}})z_1^{s_1}\cdots z_{\mathsf{N}}^{s_{\mathsf{N}}},
\end{equation}
and $ 0\leq z_1,\ldots,z_{\mathsf{N}} \leq 1$, where $p^i(s_1,\ldots,s_{\mathsf{N}})$ is the probability that a 
particle of type $i$ has $s_j$ children
of type $j$, for $j=1,\ldots,\mathsf{N}$. For $\mathsf{i}=(i_1,\ldots,i_{\mathsf{N}})$
and $\mathsf{j}=(j_1,\ldots,j_{\mathsf{N}})$, the one-step transition probabilities are given by
\begin{equation*}
\mathsf{p}(\mathsf{i},\mathsf{j})
   =\P[\mathsf{Z}_{n+1}=\mathsf{j}|\mathsf{Z}_n=\mathsf{i}] 
   = \text{coefficient of } \mathsf{z}^{\mathsf{j}}
      \text{ in } \big(\mathsf{f}(\mathsf{z})\big)^{\mathsf{i}},
\end{equation*}
where  $\big(\mathbf{f}(\mathbf{z})\big)^\mathbf{i} = \prod_{k=1}^{\mathsf{N}} f^k(\mathbf{z})^{i_k}$. For vectors $\mathsf{z},\mathsf{s}$, we write $\mathsf{z}^{\mathsf{s}}=(z_1^{s_1},\ldots,z_{\mathsf{N}}^{s_{\mathsf{N}}})$.
Let $M=(m_{ij})$ be the matrix of the first moments:

\begin{equation}
\label{eq:good_children_first_moment_matrix}
m_{ij} = \E\big[\mathsf{Z}_{1,j}\lvert \mathsf{Z}_0 = \mathbf{e}_i\big] 
= \left.\frac{\partial f^i(z_1,\ldots,z_{\mathsf{N}})}{\partial z_j}\right\rvert_{\mathsf{z}=\mathbf{1}},
\end{equation}
where $\mathbf{1} = (1,\ldots,1)$. Then $m_{ij}$ represents the expected number of offsprings of type $j$ of a particle of type $i$ in one generation.
If there exists an $n$ such that $m_{ij}^{(n)}>0$ for all $i,j$, then $M$
is called {\em strictly positive} and the process $\mathsf{Z}_n$ is called
{\em positive regular}. If each particle has exactly one child, then $\mathsf{Z}_n$
is called {\em singular}. The following is well known; see {\sc Harris} \cite{harris_book}.
\begin{theorem}\label{thm:surv/extinc}
 Assume $\mathsf{Z}_n$ is positive regular and nonsingular, and let $\rho(M)$ be the spectral radius
 of $M$. If $\rho(M)\leq 1$, then the process $\mathsf{Z}_n$ dies with probability one.
 If $\r(M)>1$, then $\mathsf{Z}_n$ survives with positive probability.
\end{theorem}

We will also make use of the mixed second moments defined as following:
\begin{equation}
\sigma^i_{jk} = \E\big[\mathsf{Z}_{1,j}\mathsf{Z}_{1,k}\lvert \mathsf{Z}_0 = \mathbf{e}_i\big]
= \left.\frac{\partial^2 f^i(z_1,\ldots,z_{\mathsf{N}})}{\partial z_j\partial z_k}\right\rvert_{\mathsf{z}=\mathbf{1}}.
\end{equation}

\subsection{Rotor walks}

Let $\T$ be a rooted tree with root $r$ and for each vertex $x$ order its neighbors counterclockwise $\{x^{0},\ldots,x^{d_x-1}\}$.  A \emph{rotor configuration} is a function
$\rho: \T\to \T$, with $\rho(x) \sim x$, for all $x\in \T$.
By abuse of notation, we write $\rho(x)=i$ if the rotor at $x$ points
to the neighbor $x^{(i)}$. A \emph{rotor walk} $(X_n)_{n\geq 0}$ is defined by the following rule.
Let $x$ be the current position of the walker, and $\rho(x) = i$ the state of
the rotor at $x$. In one step the walker does the following: it increments the rotor at $x$ to point to the next neighbor $x^{(i+1)}$ in the counterclockwise order of the neighbors of $x$, that is $\rho(x)$ is set to $i+1$ (with addition performed modulo $d_x$). Then it moves to position $x^{(i+1)}$. The rotor walk is obtained by repeatedly
applying this rule. We suppose that it starts at the root, that is $X_0=r$.

For a vertex $x\in\T$ define the set of \emph{good
children} as $\big\{x^{(k)}:\: \rho(x) < k \leq d_x\big\}$. This means that a particle performing rotor walk will first visit all its good children before visiting its ancestor. An infinite sequence of vertices
$\big(x_n\big)_{n\in\N}$ with each vertex being a child of the previous one, is called a \emph{live path} if for every
$n\geq 0$ the vertex $x_{n+1}$ is a good child of $x_n$. An {\em end} of $\T$
is an infinite sequence of vertices $x_1,x_2,\ldots$ each being the ancestor of the next.
An end is called {\em live} if the subsequence $(x_i)_{i\geq j}$ starting at one of its vertices is a live path. This definitions were introduced in \cite{angel_holroyd}, and we use them as there.

\paragraph{Nondeterministic Rotor Configurations on Directed Covers.}
Let $\mathsf{T}_i$ be a periodic tree with $\mathsf{N}$ types of vertices, and root $r$ of type $i$, to whom an additional sink vertex $o$ is added. Moreover, let $(X_n^i)$ be a rotor walk on $\mathsf{T}_i$
starting at $r$, with initial random configuration, distributed as following. 
Let  $\D=(\D_1,\ldots,\D_{\mathsf{N}})$ be a vector of probability distributions: for each $i\in\{1,\ldots,\mathsf{N}\}$,
$\D_i$ is a probability distribution with values in $\{0,\ldots,d_i\}$. Consider 
a \emph{random initial configuration $\rho$} of rotors on $\mathsf{T}_i$, such that $\left(\rho(x)\right)_{x\in\mathsf{T}_i}$
are independent random variables, and $\rho(x)$ has distribution $\D_j$ if the vertex $x$ is of type $j$.
Shortly
\begin{equation}\label{eq:random_cfg}
 \rho(x) \stackrel{d}{\sim} \D_j \quad \Longleftrightarrow \quad \tau(x)=j.
\end{equation}
If \eqref{eq:random_cfg} is satisfied,
we shall say that the rotor configuration $\rho$ is $\D=(\D_1,\ldots,\D_{\mathsf{N}})$-distributed,
and we write $\rho\stackrel{d}{\sim}\D$.

\begin{definition}
\label{defn:good_children}
For $i\in\{1,\ldots,\mathsf{N}\}$ and $k\in\{0,\ldots,d_i\}$ denote by $\mathfrak{C}_i^j(k)$ the number of good children with type $j$ of
a vertex $x$ with type $i$, if the rotor $\rho(x)$ at $x$ is in position $k$, i.e.,
\begin{equation*}
\mathfrak{C}_i^j(k) = \#\big\{l\in\{k+1,\ldots,d_i\}:\: \chi_i(l) = j\big\}.
\end{equation*}
\end{definition}
We have that
$
 \sum_{j=1}^{\mathsf{N}}\mathfrak{C}_i^j(k)=d_i-k.
$
Using this definition we can now define a MBP which models connected subtrees consisting of only good children.
In this MBP, $p^i(s_1,\ldots,s_{\mathsf{N}})$ represents the probability that a vertex of type
$i$ has $s_j$ good children of type $j$, with $j=1,\ldots,\mathsf{N}$.
Define the generating function of the MBP as in \eqref{eq:MBP_generating_function} and the 
probabilities $p^i$  by
\begin{equation}\label{eq:mbp}
p^i(s_1,\ldots,s_m) = \begin{cases}
                      \D_i(k)\; &\text{if for all $j=1,\ldots,\mathsf{N}:\  s_j=\mathfrak{C}_i^j(k)$, and $k\in\{0,\ldots,d_i\}$}, \\
                      0         &\text{otherwise},
                      \end{cases}
\end{equation}
with  $\D_i(k)=\P[\rho(x)=k]$,  for $k\in\{0,\ldots,d_i\}$ and $i\in\{1,\ldots,\mathsf{N}\}$.
In the following we always make the additional assumption that this MBP is \emph{positive
regular and nonsingular}, such that Theorem \eqref{thm:surv/extinc} can be applied. In
particular, when the rotors
point to every neighbor with positive probability these two conditions are always satisfied.
Let $M(\D)$ be the first moment matrix --- as defined in \eqref{eq:good_children_first_moment_matrix} --- of the MBP
with offspring probabilities given in \eqref{eq:mbp}, and $\rho=\rho(M(\D))$ its spectral radius.
It has been shown in \cite{huss_sava_trans_directed_covers} that
the rotor walk $(X_n^i)$ is recurrent if and only if $\rho \leq 1$, and transient if $\rho>1$.
For the MBP with offspring distribution as in  \eqref{eq:mbp}, since every vertex in $\mathsf{T}_i$ has finite degree, also the entries $\sigma_{jk}^i$  of the second moment matrix are finite,
for all $i,j,k\in\{1,\ldots,\mathsf{N}\}$.

If $\rho(M(\mathcal{D}))=1$, we call $(X_n)$ \emph{null recurrent}, and we refer to the \emph{critical case}. Otherwise, if $\rho(M(\mathcal{D}))<1$, we say that $(X_n)$ is positive recurrent.

\section{Positive recurrent rotor walks}

From now on, we let $\mathsf{T}_i$ be a periodic tree with $\mathsf{N}$ types, and root of type $i\in\{1,\ldots,\mathsf{N}\}$, and let $D$ the $\mathsf{N}\times\mathsf{N}$-adjacency matrix of the finite graph which generates $\mathsf{T}_i$. Let $(X_n^i)$ be a rotor walk on $\mathsf{T}_i$ with $\mathcal{D}$-distributed initial configuration of rotors, as defined in \eqref{eq:random_cfg}. Denote by $\mathsf{T}_i^{\mathsf{good}}$ the tree of good children of $(X_n^i)$, and the associated MBP (multitype branching process) with transition probabilities as in \eqref{eq:mbp}. We denote again by $\mathsf{Z}_n$ the size of the $n$-th generation of this MBP, that is, $\mathsf{Z}_n=(\mathsf{Z}_{n,1},\ldots,\mathsf{Z}_{n,\mathsf{N}})$, where $\mathsf{Z}_{n,i}$ represents the number of good children of type $i$ in the $n$-th generation of the MBP. Finally, let $M=M(\mathcal{D})$ be the first moment matrix of $\mathsf{Z}_n$. 

In this section we handle the case $\rho(M) < 1$, when the rotor walk $(X_n^i)$ is \emph{positive recurrent}.  
We first look a look at the range $R_n^i=\{X_1^{(i)},\ldots,X_n^i\}$ at times $(\tau_k^i)$, when the rotor walk $(X_n^i)$ returns to the sink $o$ for the $k$-th time: $\tau_0^{i} = 0$ and for $k\geq 1$ let
$$\tau_k^i = \inf\{n > \tau_{k-1}^i: X_n^i = o\}.$$
In the recurrent case, these stopping times are almost surely finite.
Let us denote by $\mathsf{R}_k^i:=R_{\tau_k^i}^{i}$. Then
\begin{equation}\label{eq:tauk-rk}
\tau^{i}_k-\tau^{i}_{k-1}=2|\mathsf{R}^i_k|.
\end{equation}
The equation above has the following explanation: at time $\tau_k^i$ the walker is at the sink, all rotors in the explored part $\mathsf{R}_k^i$ of the tree $\mathsf{T}_i$ point toward the root, while all other rotors are still in their random initial configuration. Between two consecutive stopping times $\tau_{k}^i$ and $\tau_k^i$, the walker performs a depth first search in the finite subtree induced by $\mathsf{R}_k^i$, by visiting every child of a vertex from right to left order. In a depth first search of $\mathsf{R}_k^i$ is visited exactly two times.
We first prove the following result.
\begin{theorem}
\label{thm:lln_direct_cover}
Let $\gamma$ be the spectral radius of the matrix $I+(D-I)(I-M)^{-1}$. For all $i=1,\ldots,\mathsf{N}$ we have the following strong law of large numbers at times $(\tau_k^i)$:
\begin{equation*}
\lim_{k\to \infty} \frac{ |\mathsf{R}_k^i|}{\tau_k^i} = \frac{1}{2}\left(1-\frac{1}{\gamma}\right) ,
\quad \text{a.s.}
\end{equation*}
\end{theorem}

In order to prove Theorem \ref{thm:lln_direct_cover}, we first look at the
size of the tree $\Tgood$ of good children, tree which, in view of $\rho < 1$, dies out almost surely.
Let $\mathbf{Y} = \sum_{\mathsf{n}=0}^\infty \mathsf{Z}_{\mathsf{n}}$
be the vector valued random variable counting the total number
of vertices of $\Tgood$, separately for each type.
Let $\mathbf{F}(\mathbf{z}) = \big(F^1(\mathbf{z}),\ldots,F^\mathsf{N}(\mathbf{z})\big)$ with
\begin{equation*}
F^i(\mathsf{z}) = \sum_{s_1,\ldots,s_{\mathsf{N}}\geq 0}
 \Pb\big[\mathbf{Y} = (s_1,\ldots,s_{\mathsf{N}})\lvert \mathsf{Z}_0 = \mathbf{e}_i\big] z_1^{s_1}\dots z_{\mathsf{N}}^{s_{\mathsf{N}}},
\end{equation*}
be the generating function of $\mathbf{Y}$. It follows (see e.g. \cite[13.2]{harris_book} for the case of Galton-Watson trees with just one type) that $F^i(\mathsf{z})$ satisfies the following functional equation
\begin{equation}
\label{eq:Y_functional_equation}
F^i(\mathbf{z}) = z_i\cdot f^i\big(\mathbf{F}(\mathbf{z})\big) =
                  z_i\cdot f^i\big(F^1(\mathbf{z}),\dots, F^{\mathsf{N}}(\mathbf{z})\big),
\end{equation}
for all $i\in\{1,\dots\mathsf{N}\}$, where $f(\mathsf{z})=(f^1(\mathsf{z}),\ldots,f^{\mathsf{N}}(\mathsf{z}))$ is the generating function of the MBP $\mathsf{Z}_n$, as defined in \eqref{eq:MBP_generating_function}.
Let now $V = \big(v_{ij}\big)_{i,j=1,\dots,\mathsf{N}}$ be the first
moment matrix of $\mathbf{Y}$, that is,
\begin{equation*}
v_{ij} = \E[\mathbf{Y}_j \lvert \mathsf{Z}_0 = \mathbf{e}_i\big] = \left.\frac{\partial F^i(\mathbf{z})}{\partial z_j} \right\rvert_{\mathbf{z}=\mathbf{1}}.
\end{equation*}
\begin{lemma}
\label{lem:direct_cover_total_size}
We have $V = (I-M)^{-1}$ where $I$ is the identity matrix, and the spectral radius of $V$ is given by
\begin{equation}
\rho(V) = \frac{1}{1-\rho(M)} > 1.
\end{equation}
\end{lemma}
\begin{proof}
Differentiating the functional equation \eqref{eq:Y_functional_equation}  and using that $f^i(\mathbf{1}) = 1$ and $\mathbf{F}(\mathbf{1}) = \mathbf{1}$ gives
\begin{align*}
v_{ij} &= \left. \left\{z_i \sum_{a=1}^{\mathsf{N}} \frac{\partial f^i}{\partial F^a}\big(\mathbf{F}(\mathbf{z})\big)\frac{\partial F^a}{\partial z_j}(\mathbf{z}) 
+ \delta_{ij} \cdot f^i\big(\mathbf{F}(\mathbf{z})\big)
\right\}\right\lvert_{\mathbf{z}=\mathbf{1}} \\
 &= \sum_{a=1}^{\mathsf{N}} \left.\frac{\partial f^i}{\partial z_a}(\mathbf{z})\right\lvert_{\mathbf{z}=\mathbf{1}}\left.\frac{\partial F^a}{\partial z_j}(\mathbf{z})\right\lvert_{\mathbf{z}=\mathbf{1}} 
 + \delta_{ij} \\
 &= \sum_{a=1}^{\mathsf{N}} m_{ia} v_{aj} + \delta_{ij}.
\end{align*}
Thus we get the matrix equality
$V = M\cdot V + I$, where $I$ is the identity matrix, and this implies that
\begin{equation*}
V = (I-M)^{-1}.
\end{equation*}
The inverse of $I-M$ exists, since by assumption $\rho(M) < 1$, and the 
equation for the spectral radius follows immediately.
\end{proof}
We next look at the mixed second moments of $\mathbf{Y}$, for which we prove the following.
\begin{lemma}
\label{lem:direct_cover_total_size_second_moment}
$\E\big[\mathbf{Y}_j \mathbf{Y}_k \lvert \mathsf{Z}_0 = \mathbf{e}_i] < \infty$ for all $i,j,k \in \{1,\dots,\mathsf{N} \}$.
\end{lemma}
\begin{proof}
Let
\begin{equation*}\xi^i_{jk} = \E\big[\mathbf{Y}_j \mathbf{Y}_k \lvert \mathsf{Z}_0 = \mathbf{e}_i] = \left.\frac{\partial^2 F^i(z_1,\ldots,z_{\mathsf{N}})}{\partial z_j\partial z_k}\right\rvert_{\mathbf{z}=\mathbf{1}}.
\end{equation*}
From the functional equation \eqref{eq:Y_functional_equation} we get
\begin{align*}
\xi^i_{jk} &= \left. \frac{\partial}{\partial z_k}\left\{z_i \sum_{a=1}^{\mathsf{N}} \frac{\partial f^i}{\partial F^a}\big(\mathbf{F}(\mathbf{z})\big)\frac{\partial F^a}{\partial z_j}(\mathbf{z}) 
+ \delta_{ij} \cdot f^i\big(\mathbf{F}(\mathbf{z})\big)
\right\}\right\lvert_{\mathbf{z}=\mathbf{1}} \\
&= \left\{\sum_{a=1}^{\mathsf{N}} \frac{\partial f^i}{\partial F^a}\big(\mathbf{F}(\mathbf{z})\big)\left(\delta_{ij}\frac{\partial F^a}{\partial z_j}(\mathbf{z}) + z_i\frac{\partial^2 F^a}{\partial z_j\partial z_k}(\mathbf{z}) + \delta_{ik}\frac{\partial F^a}{\partial z_k}(\mathbf{z})\right)\right. \\
&\phantom{=}\;\;\;\left.\left.+ z_i\sum_{a,b=1}^{\mathsf{N}}\frac{\partial^2 f^i}{\partial F^a \partial F^b}\big(\mathbf{F}(\mathbf{z})\big)
\frac{\partial F^a}{\partial z_j}(\mathbf{z})\frac{\partial F^b}{\partial z_k}(\mathbf{z})\right\}\right\lvert_{\mathbf{z}=\mathbf{1}}.
\end{align*}
Using the fact that $\mathbf{F}(\mathbf{1}) = \mathbf{1}$, the above equation simplifies to
\begin{align}
\label{eq:xi_equation}
\xi^i_{jk} &= \sum_{a=1}^{\mathsf{N}}m_{ia}\big(\delta_{ik} v_{aj} + \xi^a_{jk} + \delta_{ij} v_{ak}\big) + 
\sum_{a,b=1}^{\mathsf{N}} \sigma^i_{ab} v_{aj} v_{bk}.
\end{align}
For each $k$, we introduce the matrices $S_k = \big(\xi^i_{jk}\big)_{i,j = 1,\dots,\mathsf{N}}$ and $\Gamma_k = \big(\gamma^i_{jk}\big)_{i,j = 1,\dots,\mathsf{N}}$,
where 
\begin{equation*}
\gamma^i_{jk} = \sum_{a=1}^{\mathsf{N}}m_{ia}\big(\delta_{ik} v_{aj} + \delta_{ij} v_{ak}\big) + 
\sum_{a,b=1}^{\mathsf{N}} \sigma^i_{ab} v_{aj} v_{bk} < \infty.
\end{equation*}
From \eqref{eq:xi_equation} it follows that $S_k = M S_k + \Gamma_k$ and thus $S_k = (I-M)^{-1} \Gamma_k$. In particular $\xi^i_{jk} < \infty$ for all $i,j,k$, which concludes the proof.
\end{proof}

Since by Lemma \ref{lem:direct_cover_total_size} also the first moment of the
total population size exists, we get following.
\begin{corollary}
\label{cor:total_size_variance}
$\Var(\mathbf{Y}|\mathsf{Z}_0 = \mathbf{e}_i) < \infty$.
\end{corollary}

\begin{definition}
For a finite subset  $G\subset \mathsf{T}_i$ of the vertex set of $\mathsf{T}_i$,
the (multitype) cardinality of $G$ is defined as $\# G = (g_1,\dots,g_\mathsf{N})$,
where $g_k = \#\big\{v\in G:\: v \text{ has type } k\big\}$.
\end{definition}

\begin{definition}
Let $G$ be a connected subset of $\mathsf{T}_i$ containing the root. Denote
by $\partial_oG$ the set of leaves of $G$, that is,
$\partial_oG = \big\{ w \in \mathsf{T}_i\setminus G:\: \exists v\in G \text{ s.t. } v\sim w\big\}$. 
\end{definition}
We will use the following simple fact.
\begin{lemma}
\label{lem:leaves_direct_cover}
Let $G$ be a finite connected subset of $\mathsf{T}_i$ containing the root. Then
\begin{equation*}
\#\partial_oG = (D-I)\cdot \# G + \mathsf{e}_i.
\end{equation*}
\begin{proof}
Let $D = \big(d_{kl}\big)_{k,l=1,\ldots,\mathsf{N}}$ be the adjacency matrix of the finite graph which generates the periodic tree $\mathsf{T}_i$. Let $H$ be the set of children of all vertices
in $G$. By the definition of periodic trees, any vertex of type $k$ has $d_{kl}$ children of type $l$, hence $\# H = D\cdot \# G$.  We can then recover the set of leaves of $G$ by 
$$\partial_oG = H \setminus G \cup \{\text{root of }\mathsf{T}_i\},$$
and the statement of the lemma  follows immediately.
\end{proof}
\end{lemma}

We are now ready to prove Theorem \ref{thm:lln_direct_cover}.

\begin{proof}[Proof of Theorem \ref{thm:lln_direct_cover}]
Write $\mathsf{R}^{i}_k = R^{i}_{\tau^{i}_k}$ for the range of $(X_n^i)$ up to time $\tau_k^i$.
From \cite{huss_sava_trans_directed_covers}, we know that $\mathsf{R}^i_{1}$ is a
subcritical multitype Galton-Watson tree with first moment matrix $M$ of the MBP with offspring probabilities as given in \eqref{eq:mbp}.
At time $\tau^{i}_k$ the
rotor walk is at the sink, all rotors in the visited set $\mathsf{R}^i_k$ point
towards the root, while all other rotors are still in their initial $\mathcal{D}$-distributed random configuration.
Thus between times $\tau^{i}_k$ and $\tau^{i}_{k+1}$ the rotor walk performs
a depth first search of $\mathsf{R}^{i}_k \cup \partial_o \mathsf{R}^{i}_k$ visiting every child of each vertex from right to left. Whenever it reaches a leaf vertex $v\in \partial_o \mathsf{R}^{i}_k$ of type $l$,
it performs an independent a.s. finite excursion starting at $v$, excursion which has the
same distribution as $\mathsf{R}^{l}_1$.
Let $L^{i}_k = \partial_o \mathsf{R}^{i}_k$. Then, by Lemma \ref{lem:leaves_direct_cover} we have 
\begin{equation}
\label{eq:L_k}
\# L^{i}_k = (D-I)\cdot\# \mathsf{R}^{i}_k + \mathbf{e}_i.
\end{equation}
Let $\mathbf{L}_0 = \mathbf{e}_i$ if the root of the tree is of type $i$, and
let $\mathbf{L}_k=\# L^{i}_k$ be the (multitype)-cardinality of the leaves of the range after the $k$-th return to the sink, that is, the $j$-th entry in $\mathbf{L}_k$ represents the number of leaves of type $j$ in $\partial_o\mathsf{R}^{i}_k$. 
By Lemma \ref{lem:direct_cover_total_size}, we get
\begin{equation*}
\E\big[\mathbf{L}_{1,j} \lvert \mathbf{L}_0 = \mathbf{e}_i\big] = 
\big(I + (D-I)(I-M)^{-1}\big)_{ji}.
\end{equation*}
It follows that $\big(\mathbf{L}_k\big)_{k\geq 0}$ is a multitype Galton-Watson
process with first moment matrix 
\begin{equation}\label{eq:fmm-mbp}
\Gamma = \big(I + (D-I)(I-M)^{-1}\big)^{T}.
\end{equation}

We also need to check the finiteness of the second cross moments.
Let $C_i = \big(c^i_{jl}\big)_{j,l=1,\ldots,\mathsf{N}}$ be the matrix of second cross moments
$$ c^i_{jl} = \E\big[\mathbf{L}_{1,j}\mathbf{L}_{1,l}\lvert \mathbf{L}_0 = \mathbf{e}_i\big].$$

Since $\mathrm{Var}(\mathbf{L_1} \lvert \mathbf{L}_0 = \mathbf{e}_i) = C_i - \E\big[\mathbf{L}_1|\mathbf{L}_0 = \mathbf{e}_1\big] \E\big[\mathbf{L}_1^T|\mathbf{L}_0 = \mathbf{e}_1\big]$, and the first moments
of $\mathbf{L}_1$ exist by Lemma \ref{lem:direct_cover_total_size} and Lemma \ref{lem:leaves_direct_cover} it suffices to check the existence of the variance.
We have
\begin{align*}
\mathrm{Var}(\mathbf{L_1} \lvert \mathbf{L}_0 = \mathbf{e}_i\big) &=
\mathrm{Var}\big((D-I) \# \mathsf{R}^{i}_1 + \mathbf{e}_i\big) \\
&= (D-I)\mathrm{Var}\big(\# \mathsf{R}^{i}_1\big)(D-I)^T,
\end{align*}
which together with  Corollary \ref{cor:total_size_variance} implies that all matrix elements of $C_i$ are finite.
Let $\gamma > 1$ be the spectral radius of $\Gamma$ and $\mathbf{u} > 0$ be the 
corresponding Perron-Frobenius eigenvector. Since $C_i$ is
finite, Kesten-Stigum Theorem \cite{kesten1966} for the branching process $\mathbf{L}_k$ implies the existence of an almost surely positive random variable $W$ such that
\begin{equation}
\label{eq:L_convergence}
\gamma^{-k} \mathbf{L}_k \to W \mathbf{u},
\end{equation}
almost surely as $k\to\infty$.
Let $e = \{u,v\}$ where $u,v$ are vertices of the tree with $v\sim u$, be
an edge of the tree. The type $\iota(e)$ of the edge $e$ is defined as
the type of the endvertex of the edge which is further away from the sink.
For each time $n\geq 0$ we let $\bm{\psi}(n) \in \N_{\geq 0}^\mathsf{N}$ be the
vector of the number of edges of each type that are traversed by the rotor
walk up to time $n$. That is, $\bm{\psi}(n) = \big(\psi_1,\ldots,\psi_\mathsf{N}\big)$,
with
\begin{equation*}
\psi_l = \# \big\{l=1,\ldots,n:\: \iota(\{X^{i}_{l-1},X^{i}_l\}) = l\},
\end{equation*}
which satisfies $\lVert \bm{\psi}(n)\lVert_1 = n$.
Moreover, if we define $\bm{\tau}_k = \bm{\psi}(\tau_k)$, then for all $k\geq 1$
\begin{equation}
\label{eq:R_k}
\bm{\tau^{i}}_k - \bm{\tau^{i}}_{k-1} = 2 \# \mathsf{R}^{i}_k,
\end{equation}
which together with \eqref{eq:L_k} yields
\begin{align*}
(D-I)\big(\bm{\tau^{i}}_k - \bm{\tau^{i}}_{k-1}\big) + 2 \mathbf{e}_i = (D-I) 2\# \mathsf{R}^{i}_k  + \mathbf{e}_i = 2 \mathbf{L}_k,
\end{align*}
assuming the initial state $\mathbf{L}_0 = \mathbf{e}_i$ for the branching process.
Multiplying the previous equation by $\gamma^{-k}$ and using \eqref{eq:L_convergence} we get
\begin{align*}
(D-I)\left(\frac{\bm{\tau^{i}}_k}{\gamma^{k}} - \frac{1}{\gamma}\cdot\frac{\bm{\tau^{i}}_{k-1}}{\gamma^{k-1}}\right) + 2 \gamma^{-k}\mathbf{e}_i  = 2 \gamma^{-k} \mathbf{L}_k \to 2 W \mathbf{u},
\end{align*}
where the limit is almost surely as $k\to\infty$.
Writing $\bm{\sigma}_k = \frac{\bm{\tau}^{i}_k}{\gamma^k}$, the previous equation
reduces to
\begin{align*}
\bm{\sigma}_k - \frac{1}{\gamma}\bm{\sigma}_{k-1} \to 2 (D-I)^{-1} W \mathbf{u},
\end{align*}
where the limit vector $2(D-I)^{-1}W\mathbf{u}$ is almost surely positive. 
Denote by $\bm{\sigma}^\star$ the common limit of $\bm{\sigma}_k$ and $\bm{\sigma}_{k-1}$. It follows that
\begin{align*}
\bm{\sigma}^\star = \lim_{k\to\infty} \frac{\bm{\tau}^{i}_k}{\gamma^k}= 2 \left(1-\frac{1}{\gamma}\right)^{-1} (D-I)^{-1} W \mathbf{u} > 0.
\end{align*}
Therefore, also $\tau^{i}_k$ grows exponentially with rate $\gamma$, thus, the almost sure limit
\begin{equation}
\label{eq:tau_limit}
\lim_{k\to\infty}\frac{\tau^{i}_k}{\gamma^k} = \lim_{k\to\infty} \frac{\big\lVert\bm{\tau}^{i}_k\big\rVert_{1}}{\gamma^k},
\end{equation}
exists and is almost surely positive.
Now since $|\mathsf{R}^{i}_k|=\big|R^{i}_{\tau_k^i}\big| = \lVert \# \mathsf{R}^{i}_k\rVert_{1}$ the
total size of the range up to time of the $k$-th return $\tau_k^i$, from \eqref{eq:R_k} we get
$\tau^{i}_k - \tau^{i}_{k-1} = 2 |\mathsf{R}^{i}_k|$, which dividing by $\tau^{i}_k$ gives
\begin{equation*}
\frac{|\mathsf{R}^{i}_k|}{\tau^{i}_k}= \frac{1}{2}\left( 1 - \frac{\tau^{i}_{k-1}}{\tau^{i}_k} \right)= 
\frac{1}{2}\left( 1 - \frac{1}{\gamma}\cdot\frac{\tau^{i}_{k-1}}{\gamma^{k-1}}\cdot\frac{\gamma^k}{\tau^{i}_k} \right).
\end{equation*}
Using now \eqref{eq:tau_limit} we get the almost sure limit
\begin{equation*}
\lim_{k\to\infty}\frac{|\mathsf{R}^{i}_k|}{\tau^{i}_k}= \frac{1}{2}\left(1-\frac{1}{\gamma}\right),
\end{equation*}
which finishes the proof.
\end{proof}
The generalization of the previos result to all times is very similar to the similar result in case of regular trees handeled in \cite{huss_sava_range_speed_trees}.
For sake of completeness, we adapt the result to periodic trees.

\begin{proof}[Proof of Theorem \ref{thm:main-thm}.]
Write again $\mathsf{R}^{i}_k = R^{i}_{\tau^{i}_k}$ for the range of $(X_n^i)$ up to time $\tau_k^i$. For every $n$, let
\begin{equation*}
k=\max\{j: \tau_j^{i}< n\},
\end{equation*}
so that $\tau_k^i< n\leq \tau_{k+1}^{i}$ and $X_n^i\in \mathsf{R}_{k+1}^i:=R^{i}_{\tau_{k+1}}$. For each $k=1,2,\ldots,$ we partition the time intervals $(\tau_k^i,\tau_{k+1}^1]$ into finer intervals, on which the behavior of the range can be stronger controlled.
 Recall, from the proof of Theorem \ref{thm:lln_direct_cover}, that  $L_k^i=\partial_o\mathsf{R}_{k}^i$ represents the set of leaves of $\mathsf{R}_{k}^i$, and $\mathbf{L}_k=\#L_k^i$ represents the multitype cardinality of $L_k^i$, which is a multitype Galton-Watson process with first moment matrix $\Gamma$ as in \eqref{eq:fmm-mbp} and spectral radius $\gamma>1$.
We order the vertices in   $L_k^i=\partial_o\mathsf{R}_{k}^i=\{x_1,\ldots,x_{|L_k^i|}\}$ from right to left, and introduce the  
 the following two (finite) sequences of stopping times $(\eta_k(j))$ and $(\theta_k(j))$ of random length $|L_k^i|+1$, as following: let $\theta_k(0)=\tau_k^i$ and $\eta_k(|L_k^i|+1)=\tau^{i}_{k+1}$ and for $j=1,2,\ldots,|L_k^i|$
\begin{equation}
\label{eq:eta-theta}
\begin{aligned}
\eta_k(j)=&\min\{l> \theta_k(j-1):\ X^{i}_l=x_j\}\\
\theta_k(j)=&\min\{l>\eta_k(j):\ X^{i}_l=x_j\text{ and } \rho(X_l^i)=x_j^{d_{x_j}}\}
\end{aligned}
\end{equation}
That is, for each leaf $x_j$, the time $\eta_k(j)$ represents the first time the rotor walk reaches $x_j$, and $\theta_k(j)$ represents  the last time the rotor walk returns to $x_j$ after making a full excursion in the subtree rooted at $x_j$. Then
\begin{equation*}
(\tau_k^i,\tau^{i}_{k+1}]=\left\{\cup_{j=1}^{|L_k^i|+1}\big(\theta_k(j-1),\eta_k(j)\big]\right\}\cup \left\{\cup_{j=1}^{|L_k^i|}\big(\eta_k(j),\theta_k(j)\big]\right\}.
\end{equation*}
For leaves $x_j$ of type $l$, with $l=1,\ldots,\mathsf{N}$, the increments $(\theta_k(j)-\eta_k(j))$
are i.i.d random variables, and distributed according to $\tau^{l}_1$, which represents the time the rotor walk, started at the root of type $l$ of a periodic tree $\mathsf{T}_l$, needs to return to the sink for the first time. 
Once the rotor walk reaches the leaf $x_j$ for the first time at time $\eta_k(j)$, the subtree rooted at $x_j$ was never visited before by a rotor walk. Even more, the tree of good children with root $x_j$ is a subcritical multitype Galton-Watson tree, which dies out almost surely. Thus, the rotor walk on this subtree is recurrent, and it returns to $x_j$ at time $\theta_k(j)$ for the last time. Then  $\theta_k(j)-\eta_k(j)$ represents the length of this excursion which has expectation
 $\E[\tau^{l}_1]=2\E[|\mathsf{R}^l_1(x_j)|]$, and  $\mathsf{R}^l_1(x_j) $ has the same distribution as $\mathsf{R}_1^{l}$ which has the first moment matrix $M$, with $\rho(M)<1$, associated with transition probabilities \eqref{eq:mbp}.
In the time intervals $(\theta_k(j-1),\eta_k(j)]$, the rotor walk leaves the leaf $x_{j-1}$
and returns to the confluent between $x_{j-1}$ and $x_j$, from where it continues its journey until it reaches $x_j$.
Then $\eta_k(j)-\theta_k(j-1)$ is the time the rotor walk needs to reach the new leaf $x_j$ after leaving $x_{j-1}$. In this time intervals, the range does not change, since $(X^{i}_n)$ makes steps only in $\mathsf{R}^i_k$. 
Depending on the position of $X_n^i$, we shall distinguish two cases:

\emph{Case 1:} There exists a $j\in\{1,2,\ldots,|L_k^i|\}$ such that $n\in (\eta_k(j),\theta_k(j)]$.

\emph{Case 2:} There exists a $j\in\{1,2,\ldots,|L_k^i|+1\}$ such that $n\in (\theta_k(j-1),\eta_k(j)]$.

\emph{Case 1:} if $n\in (\eta_k(j),\theta_k(j)]$ for some $j$, then $\eta_k(j)<n\leq \theta_k(j)$
\begin{equation}\label{eq:case1-rec}
\dfrac{|R^{i}_{\eta_k(j)}|}{\theta_k(j)}\leq \dfrac{|R^{i}_n|}{n}\leq \dfrac{|R^{i}_{\theta_k(j)}|}{\eta_k(j)},
\end{equation}
and we show that, as $k\to \infty$, the difference $\dfrac{|R^{i}_{\theta(j)}|}{\eta(j)}-\dfrac{|R^{i}_{\eta(j)}|}{\theta(j)}\to 0$, almost surely. We use the following relations: for all $i=1,\ldots, \mathsf{N}$ and all $j=1,2,\ldots,|L_k^i|$, if the type of $x_j$ is $l\in \{1,\ldots, \mathsf{N}\} $
\begin{align*}
|R_{\theta_k(j)}^{i}|&=|R_{\eta_k(j)}^{i}|+|\mathsf{R}^l_1(x_j)|\\
\theta_k(j)& = \eta_k(j)+\tau^{l}_1(x_j),
\end{align*}
where $\tau^{l}_1(x_j)$ is a random variable which is distributed as $\tau_1^{l}$, which is finite almost surely. Moreover, $\mathsf{R}^l_1(x_j)$, represents the range of the rotor walk started at $x_j$ which has type $l$, until the first return to $x_j$; $\mathsf{R}^l_1(x_j)$ has the same distribution as $\mathsf{R}^l_1$, the range of the rotor walk started at the root $r$ of type $l$ of a periodic tree $\mathsf{T}_l$, until the first return $\tau^{l}_1$.
Then,
\begin{align*}
0&\leq \dfrac{|R^{i}_{\theta_k(j)}|}{\eta_k(j)}-\dfrac{|R^{i}_{\eta_k(j)}|}{\theta_k(j)}\\
&=\dfrac{\tau_1^{l}(x_j)|\mathsf{R}_1^{l}|}{\eta_k(j)(\eta_k(j)+\tau^{l}_1(x_j))}+\dfrac{|\mathsf{R}_1^{l}|}{\eta_k(j)+\tau^{l}_1(x_j)}+\dfrac{\tau_1^{l}(x_j)|R^{i}_{\eta_k(j)}|}{\eta_k(j)(\eta_k(j)+\tau^{l}_1(x_j))}\\
&\leq \dfrac{2|\mathsf{R}_1^{l}|}{\eta_k(j)}+\dfrac{\tau_1^{l}(x_j)|R^{i}_{\eta_k(j)}|}{(\eta_k(j))^2}.
\end{align*}
As $k\to\infty$, $\eta_k(j)\to \infty$, but $\tau^{l}_1(x_j)$ and $|\mathsf{R}_1^{l}|$ are finite, almost surely, therefore the first term in the last inequation above converges to $0$ almost surely, and we still have to prove  convergence to $0$ of the second term. But we know that $\tau_k^i<\eta_k(j)<\tau_{k+1}^{i}$, therefore
\begin{equation}\label{eq:r-eta-theta}
0\leq  \dfrac{\tau_1^{l}(x_j)|R^{i}_{\eta_k(j)}|}{(\eta_k(j))^2} \leq 
\frac{\tau_1^{l}(x_j)}{\tau_k^i}\cdot \left(\frac{|\mathsf{R}^i_{k}|}{\tau_k^i}+\frac{\sum_{l=1}^j(\theta_k(l)-\eta_k(l))}{2\tau_k^i}\right)\to 0,
\end{equation}
almost surely.
Since $\tau_1^{l}(x_j)$ is finite almost surely, and $\tau_k^i\to\infty$ as $k\to\infty$, it is clear that the 
first fraction goes to $0$, while $\frac{|\mathsf{R}_k^i|}{\tau_k^i}$ converges to $\frac{1}{2}(1-\frac{1}{\gamma})$ by Theorem \ref{thm:lln_direct_cover}. The last term in the equation above is bounded from above by $\frac{\tau_{k+1}^i-\tau_k^i}{2\tau_k^i}$ which in view of \eqref{eq:tauk-rk} and Theorem \ref{thm:lln_direct_cover} converges almost surely to $\gamma<\infty$ as $k\to\infty$. This shows that 
$$\dfrac{|R^{i}_{\theta_k(j)}|}{\eta_k(j)}-\dfrac{|R^{i}_{\eta_k(j)}|}{\theta_k(j)}\to 0,
\quad \text{almost surely, as } k\to\infty,$$ 
therefore there exists a number $\alpha_0=\lim\dfrac{|R^{i}_{\theta_k(j)}|}{\eta_k(j)}$ almost surely, and both the right and the left hand side in \eqref{eq:case1-rec} converge to $\alpha_0<\infty$ almost surely, which implies that $\frac{|R^{i}_n|}{n}$ converges to $\alpha_0$ almost surely, as well. But, along the subsequence $(\tau_k^i)$, we have from Theorem \ref{thm:lln_direct_cover}, that $\frac{|\mathsf{R}^i_k|}{\tau_k^i}\to \frac{1}{2}\left(1-\frac{1}{\gamma}\right)$ almost surely, as $k\to \infty$, therefore $\alpha_0=\frac{1}{2}\left(1-\frac{1}{\gamma}\right)$.

\emph{Case 2:} if $n\in (\theta_k(j-1),\eta_k(j)]$ for some $j\in\{1,2,\ldots,|L_k^i|+1\}$, and since in this time interval the rotor walk $X^{i}_n$ visits no new vertices and moves only in $\mathsf{R}^i_k$, we have
\begin{equation}\label{eq:case2}
\frac{|R^{i}_{\theta_k(j-1)}|}{n}=\frac{|R^i_n|}{n}=\frac{|R^{i}_{\eta_k(j)}|}{n}.
\end{equation}
In view of $\lim\dfrac{|R^{i}_{\theta_k(j)}|}{\eta_k(j)}\to \frac{1}{2}\left(1-\frac{1}{\gamma}\right)$
almost surely
and 
\begin{equation*}
\dfrac{|R^{i}_{\theta_k(j)}|}{\eta_k(j)}=\dfrac{|R^{i}_{\theta_k(j)}|}{\theta_k(j)-\tau_1^{l}(x_j)}=
\dfrac{|R^{i}_{\theta_k(j)}|}{\theta_k(j)}\frac{1}{1-\frac{\tau_1^{l}(x_j)}{\theta_k(j)}},
\end{equation*}
and $\frac{\tau_1^{l}(x_j)}{\theta_k(j)}\to 0$, it follows that $\dfrac{|R^{i}_{\theta_k(j)}|}{\theta_k(j)}$ converges almost surely to $\frac{1}{2}\left(1-\frac{1}{\gamma}\right)$. The same argument can be applied for the almost sure convergence of $\dfrac{|R^{i}_{\eta_k(j)}|}{\eta_k(j)}$ to $\frac{1}{2}\left(1-\frac{1}{\gamma}\right)$. Finally, in view of \eqref{eq:case2} 
\begin{equation*}
\dfrac{|R^{i}_{\eta_k(j)}|}{\eta_k(j)}\leq \dfrac{|R^{i}_{\eta_k(j)}|}{n}=\frac{|R^i_n|}{n}=\frac{|R^{i}_{\theta_k(j-1)}|}{n}\leq \frac{|R^{i}_{\theta_k(j-1)}|}{\theta_k(j-1)},
\end{equation*}
both the lower bound and the upper bound in the previous equation converge almost surely to $\frac{1}{2}\left(1-\frac{1}{\gamma}\right)$. Thus, also in this case $\frac{|R^{i}_n|}{n}$ converges almost surely to $\frac{1}{2}\left(1-\frac{1}{\gamma}\right)$, and the claim follows.
\end{proof}

\section{Palindromic trees}
If we look at periodic trees with a strong mirror symmetry we can give
a more geometric interpretation of the limit in Theorem \ref{thm:lln_direct_cover}. 
\begin{definition}
A periodic tree $T$ with production rule $\chi_i(k)$ is called palindromic if
the word $\big(\chi_i(1),\chi_i(2),\ldots,\chi_i(d_i)\big)$ is a palindrome
for all types $i\in\{1,\ldots,\mathsf{N}\}$, that is,
\begin{equation*}
\chi_i(k) = \chi_i(d_i + 1 - k),
\end{equation*}
for all $k\in\{1,\ldots, d_i\}$.
\end{definition}
Let $(X_n)$ be a rotor walk on $T$, and denote by $T^{\mathsf{good}}$ the tree of good children for $(X_n)$.

\begin{lemma}
\label{lem:a2m}
Let $T$ be a palindromic periodic tree, and $(X_n)$ a rotor walk with uniform initial rotor configuration. Let $D$ be the adjacency matrix of the graph which generates $T$ and $M$ the
first moment matrix of $T^{\mathsf{good}}$. Then
\begin{equation*}
D = 2 M.
\end{equation*}
\end{lemma}
\begin{proof}
Recall that $D = \big(d_{ij}\big)$ with $d_{ij} = \sum_{k=1}^{d_i} \indicator\{\chi_i(k) = j\}$. For uniformly distributed rotors the
first moment matrix $M = \big(m_{ij}\big)$ is given by
\begin{equation*}
m_{ij} = \frac{1}{d_{i+1}}\sum_{k=1}^{d_i} k \indicator\{\chi_i(k) = j \}.
\end{equation*}
Fix $i\in\{1,\dots,\mathsf{N}\}$ assuming $d_i$ is even. We can split $m_{ij}$ into two summands as follows
\begin{align*}
m_{ij} &= \frac{1}{d_i+1}\left(\sum_{k=1}^{\frac{d_i}{2}} k \indicator\{\chi_i(k) = j \} + \sum_{k=\frac{d_i}{2}+1}^{d_i} k \indicator\{\chi_i(k) = j \}\right) \\
&= \frac{1}{d_i+1}\left(\sum_{k=1}^{\frac{d_i}{2}} k \indicator\{\chi_i(k) = j \} + \sum_{k=\frac{d_i}{2}+1}^{d_i} k \indicator\{\chi_i(d_i + 1 - k) = j \}\right), \\
\end{align*}
where in the second line we use that $T$ is palindromic. Changing the order of summation of the second sum and using the
substitution $l = d_i + 1 - k$ gives
\begin{align*}
m_{ij} &= \frac{1}{d_i+1}\left(\sum_{k=1}^{\frac{d_i}{2}} k \indicator\{\chi_i(k) = j \} + \sum_{l=1}^{\frac{d_i}{2}} (d_i + 1 - l) \indicator\{\chi_i(l) = j \}\right) \\
&= \frac{1}{d_i+1}\sum_{k=1}^{\frac{d_i}{2}} (k + d_i +1 - k) \indicator\{\chi_i(k)=j\} \\
&= \sum_{k=1}^{\frac{d_i}{2}} \indicator\{\chi_i(k)=j\} = \frac{1}{2} d_{ij},
\end{align*}
where the last identity is again due to the palindromic property of $T$.
Next assume $d_i$ is odd. We rewrite $m_{ij}$ into three summands
\begin{multline*}
m_{ij} = \frac{1}{d_i+1}\Bigg(\sum_{k=1}^{\frac{d_i-1}{2}} k \indicator\{\chi_i(k) = j \} \\ + \frac{d_i+1}{2}\indicator\Big\{\chi_i\Big(\frac{d_i+1}{2}\Big) = j \Big\} +\sum_{k=\frac{d_i+3}{2}}^{d_i} k \indicator\{\chi_i(k) = j \}\Bigg).
\end{multline*}
Again using the palindromic property of $T$, and using the substitution
$l = d_i + 1 -k$ we can transform the third summand in the last equation
\begin{align*}
\sum_{k=\frac{d_i+3}{2}}^{d_i} k \indicator\{\chi_i(k) = j \} &=
\sum_{k=\frac{d_i+3}{2}}^{d_i} k \indicator\{\chi_i(d_i + 1 - k) = j \}  \\
&=\sum_{l=1}^{\frac{d_i-1}{2}} (d_i + 1 -l) \indicator\{\chi_i(l) = j \}.
\end{align*}
Thus
\begin{align*}
	m_{ij} &= \frac{1}{2} \indicator\Big\{\chi_i\Big(\frac{d_i+1}{2}\Big) = j \Big\} + \frac{1}{d_i + 1}\sum_{k=1}^{\frac{d_i-1}{2}} (k + d_i + 1 -k) \indicator\{\chi_i(k) = j \} \\
	&= \frac{1}{2} \indicator\Big\{\chi_i\Big(\frac{d_i+1}{2}\Big) = j \Big\} + \sum_{k=1}^{\frac{d_i-1}{2}} \indicator\{\chi_i(k) = j \}
	= \frac{1}{2} d_{ij},
\end{align*}
where the last equality is again due to the palindromic property. It follows
that $2M = D$.
\end{proof}

The following is obvious:
\begin{lemma}
\label{lem:rho_a2m_relation}
Let $D$ be a $(\mathsf{N}\times\mathsf{N})$-matrix with spectral radius $\psi$. Let $\gamma$ be the spectral radius of $I+(D-I)(I-\alpha^{-1}\cdot D)^{-1}$, for some real number $\alpha$. If the spectral radius $\psi$ of $D$ is
not equal to $\alpha$ then
\begin{equation*}
\gamma = 1+\frac{\psi - 1}{1 - \psi/\alpha} = \frac{(\alpha-1)\psi}{\alpha-\psi}.
\end{equation*}
\end{lemma}

\begin{theorem}
Let $T$ be a palindromic periodic tree, and consider a rotor walk $(X_n)$ on $T$ with uniform 
initial rotor configuration. Let $\mathsf{br}(T)$ be the branching
number of $T$. Then $(X_n)$ is recurrent if
and only if $\mathsf{br}(T) \leq 2$ and in the positive recurrent case (when $\mathsf{br}(T)<2$)
we have: 
\begin{equation*}
\lim_{n\to \infty} \frac{|R_n|}{n} = \frac{\mathsf{br}(T) -1}{\mathsf{br}(T)} ,
\quad \text{a.s.},
\end{equation*}
\end{theorem}
\begin{proof}
Recall that the branching number of $T$ is equal to the spectral radius of the matrix $D$. By Lemma \ref{lem:a2m} $D = 2M$. In the
positive recurrent case, we can
apply Lemma \ref{lem:rho_a2m_relation} with $\alpha=2$, which gives 
$\gamma = \frac{\mathbf{br(T)}}{2-\mathbf{br(T)}}$, which together with
Theorem \ref{thm:lln_direct_cover} completes the proof.
\end{proof}

\begin{appendices}

\section{Rotor-recurrent trees}

\begin{definition}
We call a tree $\mathsf{T}$ \emph{rotor-recurrent} if the rotor walk $(X_n)$ on $\mathsf{T}$ with uniform initial rotor configuration is recurrent. If $(X_n)$ is transient on $\mathsf{T}$, we call $\mathsf{T}$ \emph{rotor-transient}.
\end{definition}

Not much is known about the recurrence and transience of rotor walks on graphs other than trees.
But even on trees we can give examples of unusual properties of
rotor-recurrence that show that a general theory of rotor-recurrence will necessarily
be much more involved than the theory of recurrence of random walks.
In \cite{huss_sava_trans_directed_covers} the authors give an example of a tree that
is rotor-recurrent or rotor-transient depending on its planar embedding into the plane.
This suggests that the underlying graph for the rotor walk does not provide enough information for
stating a rotor-recurrence criteria. The whole ribbon structure of the graph, which determines
the way the rotors turn, will be needed for that.
The rotor-recurrence has even more peculiar properties as the following example shows.

\begin{figure}
	\centering
	\input{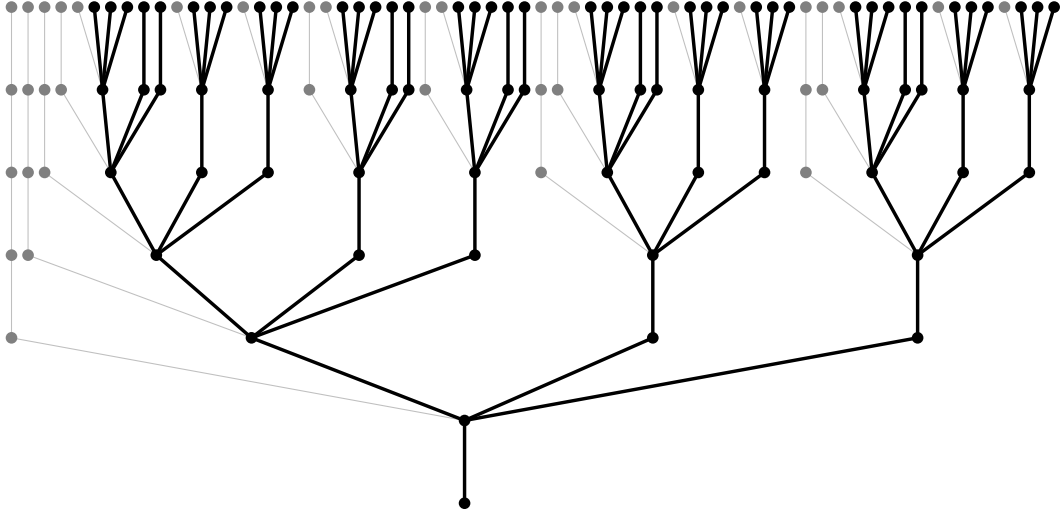}
	\caption{A rotor-recurrent tree that contains a transient subtree, which is drawn in black.}
\end{figure}

\begin{example}
Let $\mathsf{T}$ be the direct cover given by the following generator $\chi$. We use the notation of
\cite{huss_sava_trans_directed_covers}; note the planar embedding of the tree, and
thus the rotor-mechanism, already specified by the table $\chi$.
\begin{center}
\begin{tabular}{lc|cccc}
	\multicolumn{2}{c|}{} & \multicolumn{4}{c}{\scalebox{0.7}{$k\rightarrow$}} \\
	\multicolumn{2}{c|}{$\chi_i(k)$} & 1 & 2 & 3 & 4\\
	\hline
	\multirow{5}{*}{\scalebox{0.7}{\rotatebox{-90}{\rotatebox{90}{$i$} $\rightarrow$}}}
	& 1 & 2 & 2 & 1 & 3\\
	& 2 & 1 \\
	& 3 & 4 \\
	& 4 & 5 \\
	& 5 & 2
\end{tabular}
\end{center}
The rotor walk on $\mathsf{T}$ has the first moment matrix $M$ given by
\begin{equation*}
M = \begin{pmatrix}
	\frac{3}{5} & \frac{3}{5} & \frac{4}{5} & 0 & 0 \\
	\frac{1}{2} & 0 & 0 & 0 & 0 \\
	0 & 0 & 0 & \frac{1}{2} & 0 \\
	0 & 0 & 0 & 0 & \frac{1}{2} \\
	0 & \frac{1}{2} & 0 & 0 & 0
\end{pmatrix},
\end{equation*}
with spectral radius $\rho(M) = 0.967$. By \cite[Theorem 3.5]{huss_sava_trans_directed_covers}, since $\rho(M)<1$, $\mathsf{T}$ is rotor-recurrent.
We now construct a subtree $\bar{\mathsf{T}}$ from $\mathsf{T}$ by deleting all vertices of type $3$ and
all their descendants. Thus $\bar{\mathsf{T}}$ is the direct cover defined by the generator $\bar{\chi}$ as follows:
\begin{center}
	\begin{tabular}{lc|ccc}
		\multicolumn{2}{c|}{} & \multicolumn{3}{c}{\scalebox{0.7}{$k\rightarrow$}} \\
		\multicolumn{2}{c|}{$\bar{\chi}_i(k)$} & 1 & 2 & 3\\
		\hline
		\multirow{2}{*}{\scalebox{0.7}{\rotatebox{-90}{\rotatebox{90}{$i$} $\rightarrow$}}}
		& 1 & 2 & 2 & 1\\
		& 2 & 1 \\
	\end{tabular}
\end{center}
The rotor walk on $\bar{\mathsf{T}}$ has first moment matrix $\bar{M}$ given by
\begin{equation*}
\bar{M} = \begin{pmatrix}
\frac{3}{4} & \frac{3}{4} \\
\frac{1}{2} & 0
\end{pmatrix},
\end{equation*}
which has spectral radius $\rho(\bar{M}) = 1.093$. Hence, by \cite[Theorem 3.5]{huss_sava_trans_directed_covers}, $\bar{\mathsf{T}}$ is rotor-transient.
\end{example}

\end{appendices}

\begin{remark}
There exists rotor-recurrent graphs that contain rotor-transient subgraphs.
\end{remark}

\bibliography{range}{}
\bibliographystyle{alpha_arxiv}

\textsc{Wilfried Huss}
\texttt{husswilfried@gmail.com};

\textsc{Ecaterina Sava-Huss, Department of Mathematics, Innsbruck University, Austria.}
\texttt{Ecaterina.Sava-Huss@uibk.ac.at};

\end{document}